\documentclass[a4paper,11pt,leqno]{amsart}
\usepackage{amssymb, amsmath, amsthm}
\usepackage[all]{xy}

\newtheorem{thm}{Theorem}[section]
\newtheorem{crl}[thm]{Corollary}

\newtheorem{prp}[thm]{Proposition}

\theoremstyle{definition}
\newtheorem{dfn}[thm]{Definition}
\newtheorem{exa}[thm]{Example}

\theoremstyle{remark}
\newtheorem*{rem}{Remark}

\numberwithin{equation}{section}

\newcommand{\Z}{\mathbb{Z}}
\newcommand{\Q}{\mathbb{Q}}
\newcommand{\R}{\mathbb{R}}
\newcommand{\bk}{\mathbf{k}}
\newcommand{\frH}{\mathfrak{H}}
\newcommand{\frP}{\mathfrak{P}}
\newcommand{\abs}[1]{{\lvert #1 \rvert}}

\DeclareMathOperator{\dep}{\mathrm{dep}}

\title{Multiple zeta-star values and multiple integrals}
\author{Shuji Yamamoto}
\thanks{This work was supported in part by 
JSPS Grant-in-Aid for Young Scientists (S) (No.\ 21674001). }  
\address{Department of Mathematics, Faculty of Science and Technology, 
Keio University, 3-14-1 Hiyoshi, Kohoku-ku, Yokohama, 223-8522, JAPAN}
\email{yamashu@math.keio.ac.jp}

\begin{document}

\begin{abstract}      
We prove a kind of integral expressions 
for finite multiple harmonic sums and multiple zeta-star values. 
Moreover, we introduce a class of multiple integrals, 
associated with some combinatorial data (called 2-labeled posets). 
This class includes both multiple zeta and zeta-star values 
of Euler-Zagier type, and also several other types of multiple zeta values. 
We show that these integrals can be used to obtain 
some relations among such zeta values quite transparently. 
\end{abstract}

\maketitle

\section{Integral expression of finite multiple harmonic sums}
We begin with the finite multiple harmonic sums 
\[s_\bk(N)=\sum_{N=m_1\geq\cdots\geq m_n\geq 1}
\frac{1}{m_1^{k_1}\cdots m_n^{k_n}}, \]
where $\bk=(k_1,\ldots,k_n)$ is an $n$-tuple of positive integers 
and $N$ is a positive integer. 
When $k_1\geq 2$, by definition, their sum gives 
the multiple zeta-star values (MZSVs for short): 
\begin{equation}\label{eq:MZSV_finite}
\zeta^\star(\bk)
=\sum_{m_1\geq\cdots\geq m_n\geq 1}\frac{1}{m_1^{k_1}\cdots m_n^{k_n}}
=\sum_{N=1}^\infty s_\bk(N). 
\end{equation}

One of the basic properties of these finite multiple sums 
is the following relation, called the duality: 

\begin{thm}[\cite{H,K}]\label{thm:duality}
For any index $\bk\in(\Z_{\geq 1})^n$ and $N\geq 0$, we have
\begin{equation}\label{eq:duality}
\sum_{i=0}^{N-1}(-1)^i\binom{N-1}{i} s_\bk(i+1)=s_{\bk^*}(N). 
\end{equation}
Here $\bk^*$ denotes the `transpose' of $\bk$ (see below). 
\end{thm}

For an index $\bk=(k_1,\ldots,k_n)$, we set 
\[\lvert\bk\rvert:=k_1+\cdots+k_n, \quad 
A(\bk):=\{k_1,k_1+k_2,\ldots,k_1+\cdots+k_{n-1}\}.\] 
Then the transpose $\bk^*$ is the index determined by the property 
\[\lvert\bk\rvert=\lvert\bk^*\rvert,\quad \{1,\ldots,\lvert\bk\rvert-1\}
=A(\bk)\amalg A(\bk^*). \] 
For example, the transpose of $(2,3)$ is $(1,2,1,1)$. 
It can be illustrated by the following picture: 
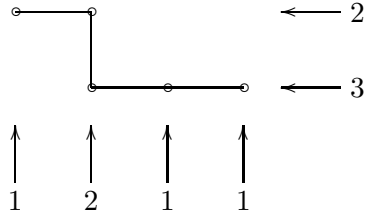
\[\begin{xy}
{(0,0) \ar @{o-o} (10,0)},
{(10,0) \ar @{-o} (10,-10)},
{(10,-10) \ar @{-o} (20,-10)},
{(20,-10) \ar @{-o} (30,-10)},
{(45,0)*+{2} \ar (35,0)},
{(45,-10)*+{3} \ar (35,-10)},
{(0,-25)*+{1} \ar (0,-15)},
{(10,-25)*+{2} \ar (10,-15)},
{(20,-25)*+{1} \ar (20,-15)},
{(30,-25)*+{1} \ar (30,-15)},
\end{xy}\]

The identity \eqref{eq:duality} is somewhat analogous to 
the well-known duality
\begin{equation}\label{eq:duality_MZV}
\zeta(a_1+1,\underbrace{1,\ldots,1}_{b_1-1},\ldots,
a_s+1,\underbrace{1,\ldots,1}_{b_s-1})
=\zeta(b_s+1,\underbrace{1,\ldots,1}_{a_s-1},\ldots,
b_1+1,\underbrace{1,\ldots,1}_{a_1-1}) 
\end{equation}
of the multiple zeta values (MZVs)
\[\zeta(\bk)=\sum_{m_1>\cdots>m_n>0}\frac{1}{m_1^{k_1}\cdots m_n^{k_n}}. \]
Since the latter duality follows immediately from 
the iterated integral expression 
\begin{equation}\label{eq:MZV}
\begin{split}
&\zeta(a_1+1,\underbrace{1,\ldots,1}_{b_1-1},\ldots,
a_s+1,\underbrace{1,\ldots,1}_{b_s-1})\\
&=\int_0^1\underbrace{\frac{dt}{t}\circ\cdots\circ\frac{dt}{t}}_{a_1}
\circ\underbrace{\frac{dt}{1-t}\circ\cdots\circ\frac{dt}{1-t}}_{b_1}
\circ\cdots\circ
\underbrace{\frac{dt}{t}\circ\cdots\circ\frac{dt}{t}}_{a_s}
\circ\underbrace{\frac{dt}{1-t}\circ\cdots\circ\frac{dt}{1-t}}_{b_s}, 
\end{split}
\end{equation}
it is natural to ask for a similar integral expression of 
finite multiple sum $s_\bk(N)$ from which \eqref{eq:duality} follows. 
Here is an answer: 

\begin{thm}\label{thm:finite}
Let $\bk=(k_1,\ldots,k_n)$ be an index, and put $k=\abs{\bk}=k_1+\cdots+k_n$. 
Moreover, define 
\begin{align*}
J(\bk)&=\{0,k_1,k_1+k_2,\ldots,k_1+\cdots+k_{n-1}\}=A(\bk)\cup\{0\},\\
\Delta(\bk)&=\Biggl\{(t_1,\ldots,t_k)\in[0,1]^k\Biggm|
\begin{array}{l}
t_j<t_{j+1} \text{ if $j\notin J(\bk)$},\\ 
t_j>t_{j+1} \text{ if $j\in J(\bk)$}
\end{array}\Biggr\}. 
\end{align*}
Then, for $N\geq 1$, we have 
\begin{equation}\label{eq:finite}
s_\bk(N)=\int_{\Delta(\bk)}t_1^{N-1}dt_1 
\omega_{\delta(2)}(t_2)\cdots\omega_{\delta(k)}(t_k), 
\end{equation}
where $\omega_0(t)=\frac{dt}{t}$, $\omega_1(t)=\frac{dt}{1-t}$ and
\begin{equation}\label{eq:delta}
\delta(j)=\begin{cases}
0 & (j-1\notin J(\bk)), \\
1 & (j-1\in J(\bk)). 
\end{cases}
\end{equation}
\end{thm}

\begin{rem}
We include $0$ in the set $J(\bk)$ to set $\delta(1)=1$ in \eqref{eq:delta}, 
though this value is not used in the above theorem. 
We need it in Corollary \ref{cor:MZSV}.  
\end{rem}

\begin{proof}
We consider the case $\bk=(2,1,2)$ as an example, 
and then the general pattern will be understood. 
In this case, we should prove 
\[s_{(2,1,2)}(N)=\int_{t_1<t_2>t_3>t_4<t_5}t_1^{N-1}dt_1\frac{dt_2}{t_2}
\frac{dt_3}{1-t_3}\frac{dt_4}{1-t_4}\frac{dt_5}{t_5}. \]
(Here, implicitly, the inequalities $0\leq t_i\leq 1$ are assumed.) 
The right-hand side is computed by repeating single integrals, i.e., 
\begin{align*}
\int_0^{t_2}t_1^{N-1}dt_1&=\frac{t_2^N}{N},\\
\frac{1}{N}\int_{t_3}^1t_2^N\frac{dt_2}{t_2}&=\frac{1-t_3^N}{N^2},\\
\frac{1}{N^2}\int_{t_4}^1(1-t_3^N)\frac{dt_3}{1-t_3}
&=\sum_{N\geq m\geq 1}\frac{1-t_4^m}{N^2m},\\
\sum_{N\geq m\geq 1}\frac{1}{N^2m}\int_0^{t_5}(1-t_4^m)\frac{dt_4}{1-t_4}
&=\sum_{N\geq m\geq l\geq 1}\frac{t_5^l}{N^2ml},\\
\sum_{N\geq m\geq l\geq 1}\frac{1}{N^2ml}\int_0^1t_5^l\frac{dt_5}{t_5}
&=\sum_{N\geq m\geq l\geq 1}\frac{1}{N^2ml^2}. 
\end{align*}
The last sum is exactly $s_{(2,1,2)}(N)$. 
\end{proof}

To deduce Theorem \ref{thm:duality} from Theorem \ref{thm:finite}, 
note that there is a bijection 
\begin{equation}\label{eq:s=1-t}
\Delta(\bk)\ni(t_1,\ldots,t_k)\longmapsto(1-t_1,\ldots,1-t_k)
\in\Delta(\bk^*)
\end{equation}
and that the map $\delta^*$ associated with $\bk^*$ as in \eqref{eq:delta} 
satisfies $\delta^*(j)=1-\delta(j)$ for $j=2,\ldots,k$. 
Hence, by changing of the integral variables $s_j=1-t_j$, we obtain 
\begin{align*}
s_{\bk^*}(N)
&=\int_{\Delta(\bk^*)}s_1^{N-1}ds_1\,\omega_{\delta^*(2)}(s_2)
\cdots\omega_{\delta^*(k)}(s_k)\\
&=\int_{\Delta(\bk)}(1-t_1)^{N-1}dt_1\omega_{\delta(2)}(t_2)
\cdots\omega_{\delta(k)}(t_k)\\
&=\sum_{i=0}^{N-1}(-1)^i\binom{N-1}{i}
\int_{\Delta(\bk)}t_1^idt_1\omega_{\delta(2)}(t_2)
\cdots\omega_{\delta(k)}(t_k)\\
&=\sum_{i=0}^{N-1}(-1)^i\binom{N-1}{i}s_{\bk}(i+1). 
\end{align*}

By \eqref{eq:MZSV_finite} and \eqref{eq:finite}, 
we also obtain an integral expression of MZSVs. 

\begin{crl}\label{cor:MZSV}
In the same notation as in Theorem \ref{thm:finite}, 
assume $k_1\geq 2$. Then 
\begin{equation}\label{eq:MZSV}
\zeta^\star(\bk)=\int_{\Delta(\bk)}
\omega_{\delta(1)}(t_1)\cdots\omega_{\delta(k)}(t_k). 
\end{equation}
\end{crl}

\begin{exa}
For $\bk=(2,1)$, we have 
\begin{equation}\label{eq:MZSV(2,1)}
\zeta^\star(2,1)=\int_{t_1<t_2>t_3}
\frac{dt_1}{1-t_1}\frac{dt_2}{t_2}\frac{dt_3}{1-t_3}. 
\end{equation}
This integral can be divided into two parts: 
\[\int_{t_1<t_2>t_3}\frac{dt_1}{1-t_1}\frac{dt_2}{t_2}\frac{dt_3}{1-t_3}
=\biggl(\int_{t_1<t_3<t_2}+\int_{t_3<t_1<t_2}\biggr)
\frac{dt_1}{1-t_1}\frac{dt_2}{t_2}\frac{dt_3}{1-t_3}. \]
By the iterated integral expression \eqref{eq:MZV} of MZVs, 
the right-hand side is equal to $\zeta(2,1)+\zeta(2,1)$. 
Therefore, from the integral expression \eqref{eq:MZSV(2,1)}, 
one obtains 
\[\zeta^\star(2,1)=2\zeta(2,1). \]
Note that this is different from the relation 
\[\zeta^\star(2,1)=\zeta(2,1)+\zeta(3) \] 
obtained from the series expressions. 
By comparing these two relations, 
one proves Euler's famous relation $\zeta(2,1)=\zeta(3)$. 

More generally, 
from the integral and series expressions of $\zeta^\star(k-1,1)$, 
one can show the sum formula for double zeta values 
\[\zeta(k-1,1)+\zeta(k-2,2)+\cdots+\zeta(2,k-2)=\zeta(k) 
\quad (k\geq 3)\]
in a similar manner. 
\end{exa}

\section{Multiple integrals associated with 2-labeled finite posets}
Now we define a class of integrals 
which includes both MZVs \eqref{eq:MZV} and MZSVs \eqref{eq:MZSV}. 
Recall that a finite poset is a finite set endowed with a partial order. 
In the following, we omit the word `finite' since we only consider 
finite posets. 

\begin{dfn}
\begin{enumerate}
\item A \emph{2-labeled poset} is a pair $X=(X,\delta_X)$ consisting of 
a poset $X$ and a map $\delta_X\colon X\to\{0,1\}$, 
called the \emph{labeling map}. 
The \emph{weight}, denoted by $\abs{X}$, 
is the number of elements of the underlying set $X$, 
and the \emph{depth}, denoted by $\dep(X)$, 
is the number of $x\in X$ such that $\delta(x)=1$. 
\item A 2-labeled poset $X$ is said \emph{admissible} 
if $\delta_X(x)=1$ for all minimal $x\in X$ and 
$\delta_X(x)=0$ for all maximal $x\in X$. 
\item For any poset $X$, we put 
\[\Delta(X):=\bigl\{(t_x)_{x\in X}\in[0,1]^X\bigm| 
t_x<t_y \text{ if } x<y\bigr\}. \]
\item For an admissible 2-labeled poset $X$, 
we define the associated integral by 
\begin{equation}\label{eq:I}
I(X):=\int_{\Delta(X)}\prod_{x\in X}\omega_{\delta_X(x)}(t_x). 
\end{equation}
Here $\omega_0(t)=\frac{dt}{t}$ and $\omega_1(t)=\frac{dt}{1-t}$ are 
the same notation as in Theorem \ref{thm:finite}. 
\end{enumerate}
\end{dfn}

\begin{rem}
For the empty 2-labeled poset, denoted $\varnothing$, 
we put $I(\varnothing)=1$. This is compatible with the usual definition 
$\zeta(\varnothing)=\zeta^\star(\varnothing)=1$, 
where $\varnothing$ denotes the index of length $0$. 
\end{rem}

We use Hasse diagrams to indicate 2-labeled posets, 
with vertices $\circ$ and $\bullet$ corresponding to $\delta(x)=0$ and $1$, 
respectively. For example, 
\[X=\{1<2<3<4<5\}\text{ and }
\bigl(\delta(1),\ldots,\delta(5)\bigr)=(1,0,1,0,0)\]
is represented as the diagram 
\[\begin{xy}
{(0,8) \ar @{o-o} (0,4)},
{(0,4) \ar @{-{*}} (0,0)},
{(0,0) \ar @{-o} (0,-4)},
{(0,-4) \ar @{-{*}} (0,-8)}
\end{xy} \]
This 2-labeled poset is admissible, and we have 
\[I(X)=\int_{t_1<t_2<t_3<t_4<t_5}
\frac{dt_1}{1-t_1}\frac{dt_2}{t_2}\frac{dt_3}{1-t_3}
\frac{dt_4}{t_4}\frac{dt_5}{t_5}
=\zeta(3,2). \]
In general, the iterated integral expression of a MZV 
is equal to the integral $I(X)$ 
associated with an admissible 2-labeled \emph{totally ordered} set $X$, 
and the weight and the depth of the MZV coincide with those of $X$. 

Another example: Corollary \ref{cor:MZSV} for $\bk=(2,3)$ gives 
\[\zeta^\star(2,3)=I\left(\ 
\begin{xy}
{(0,-4) \ar @{{*}-o} (4,0)},
{(4,0) \ar @{-{*}} (8,-4)},
{(8,-4) \ar @{-o} (12,0)},
{(12,0) \ar @{-o} (16,4)}
\end{xy}\ \right). \]

Here we collect some basic constructions on 2-labeled posets, 
and relations between the associated integrals.

\begin{dfn}
\begin{enumerate}
\item For 2-labeled posets $X$ and $Y$, 
one can naturally define their direct sum $X\amalg Y$: 
Its underlying poset is the direct sum of finite sets $X$ and $Y$ 
endowed with the partial order 
\begin{align*}
x\leq y \text{ in $X\amalg Y$}\iff 
&x,y\in X\text{ and }x\leq y\text{ in $X$ or }\\
&x,y\in Y\text{ and }x\leq y\text{ in $Y$.}
\end{align*}
The map $\delta_{X\amalg Y}\colon X\amalg Y\to\{0,1\}$ 
is the direct sum of the maps $\delta_X\colon X\to\{0,1\}$ and 
$\delta_Y\colon Y\to\{0,1\}$. 
\item Let $X=(X,\leq)$ be a poset, and $a,b\in X$ not comparable, i.e., 
neither $a\leq b$ nor $b\leq a$ hold. 
Then we denote by $X_a^b$ the poset with the same underlying set $X$ 
and endowed with the order $\leq_a^b$ defined by 
\[x\leq_a^b y\iff 
x\leq y,\text{ or } x\leq a\text{ and }b\leq y. \]
We call $X_a^b$ the refinement of $X$ obtained by imposing $a<b$. 
If $X$ is a 2-labeled poset, then $X_a^b$ also becomes 
a 2-labeled poset with the same labeling map. 
\item For a 2-labeled poset $X$, we define its \emph{transpose} $X^*$ 
as the 2-labeled poset consisting of the same underlying set as $X$ 
endowed with the reversed order (i.e.\ $x\leq y$ in $X^*$ if and only if 
$y\leq x$ in $X$), and the labeling map $\delta_{X^*}(x)=1-\delta_X(x)$. 
\end{enumerate}
\end{dfn}

\begin{prp}
\begin{enumerate}
\item If $X$ and $Y$ are admissible 2-labeled posets, 
then $X\amalg Y$ is admissible and 
\begin{equation}\label{eq:prod}
I(X\amalg Y)=I(X)\cdot I(Y). 
\end{equation}
\item If $X$ is an admissible 2-labeled poset, and $a$ and $b\in X$ are 
not comparable, then both $X_a^b$ and $X_b^a$ are admissible and 
\begin{equation}\label{eq:sum}
I(X)=I(X_a^b)+I(X_b^a). 
\end{equation}
\item If $X$ is an admissible 2-labeled poset, 
then the transpose $X^*$ is admissible and 
\begin{equation}\label{eq:dual}
I(X^*)=I(X). 
\end{equation}
\end{enumerate}
\end{prp}
\begin{proof}
All assertions are easily verified. 
Note that \eqref{eq:dual} is shown by making the change of variables 
\[\Delta(X)\ni(t_x)\longmapsto (1-t_x)\in \Delta(X^*), \]
which is a generalization of \eqref{eq:s=1-t}. 
\end{proof}

\begin{rem}
The shuffle relation for the MZVs can be derived from 
the identities \eqref{eq:prod} and \eqref{eq:sum} 
(see also the remark after Corollary \ref{cor:sum of MZVs}). 
On the other hand, the identity \eqref{eq:dual} is a natural generalization 
of the duality \eqref{eq:duality_MZV} for the MZVs. 
\end{rem}

\begin{crl}\label{cor:sum of MZVs}
For any 2-labeled poset $X$, the integral $I(X)$ can be expressed 
as the sum of a finite number of MZVs of weight $\abs{X}$ 
and depth $\dep(X)$. 
\end{crl}
\begin{proof}
By using \eqref{eq:sum} several times, we express $I(X)$ 
as a sum of the integrals associated with 2-labeled totally ordered sets, 
i.e., the integral expressions of MZVs. 
Each of these 2-labeled totally ordered sets consists of 
the same underlying set and labeling map as $X$ and 
a total order extending the partial order of $X$. 
In particular, these have the same weight and depth as $X$. 
\end{proof}

\begin{rem}
There is an algebraic formalism for the integrals $I(X)$, 
similar to the well-known pair of 
the shuffle algebra $\frH=\Q\langle x,y\rangle$ 
and the homomorphism $Z\colon\frH^0=\Q\oplus x\frH y\to\R$. 

We write $\frP$ for the $\Q$-vector space generated by 
all isomorphism classes of 2-labeled posets, 
and define a product on it by $[X]\cdot[Y]=[X\amalg Y]$. 
Then, the subspace $\frP^0$ generated by admissible 2-labeled posets 
is a subalgebra, and the map $I\colon \frP^0\to\R$, 
defined by linearity, is indeed a $\Q$-algebra homomorphism. 

In fact, there exists a surjective $\Q$-algebra homomorphism 
$\rho\colon\frP\to\frH$ satisfying 
$\rho(\frP^0)=\frH^0$ and $Z\circ\rho=I$. 
This is defined as the unique homomorphism 
whose kernel is generated by $[X]-[X_a^b]-[X_b^a]$ 
and which sends totally ordered $[X]$ to a monomial in $\frH$ 
encoding $\delta_X$ appropriately. 
\end{rem}

\section{Other examples}
In this section, we consider some values representable 
by the integrals $I(X)$, other than MZVs and MZSVs. 

\subsection{Arakawa-Kaneko zeta values}
The Arakawa-Kaneko multiple zeta function \cite{AK} is defined as 
\[\xi_k(s)=\frac{1}{\Gamma(s)}\int_0^\infty
\frac{Li_k(1-e^{-t})}{1-e^{-t}}e^{-t}t^{s-1}dt\]
for an integer $k>0$, where $Li_k(x)=\sum_{n=1}^\infty\frac{x^n}{n^k}$ 
is the $k$-th polylogarithm function. 
By making the variable change $x=1-e^{-t}$, we can write it as 
\[\xi_k(s)=\frac{1}{\Gamma(s)}\int_0^1
Li_k(x)\bigl(-\log(1-x)\bigr)^{s-1}\frac{dx}{x}. \]
Now we use integral expressions 
\[Li_k(x)=\int_{x>t_1>\cdots>t_k>0}
\frac{dt_1}{t_1}\cdots\frac{dt_{k-1}}{t_{k-1}}\frac{dt_k}{1-t_k} \]
and 
\[-\log(1-x)=\int_0^x\frac{du}{1-u}, \]
to deduce for any integer $n>0$ 
\begin{equation}\label{eq:AK}
\xi_k(n)=\frac{1}{(n-1)!}\,I\left(\xybox{
{(0,-9) \ar @{{*}-o} (0,-4)}, 
{(0,-4) \ar @{.o} (0,4)}, 
{(0,4) \ar @{-o} (10,9)}, 
{(10,9) \ar @{-{*}} (6,-5)}, 
{(10,9) \ar @{-{*}} (16,-5)}, 
{(8,-5) \ar @{.} (14,-5)}, 
{(-1,-9) \ar @/^1mm/ @{-} ^k (-1,4)}, 
{(6,-6) \ar @/_1mm/ @{-} _{n-1} (16,-6)}, 
}\ \right). 
\end{equation}
Moreover, there are exactly $(n-1)!$ ways to impose 
a total order on the $n-1$ black vertices. 
Thus we have the identity 
\[\xi_k(n)=I\left(\xybox{
{(0,-9) \ar @{{*}-o} (0,-4)}, 
{(0,-4) \ar @{.o} (0,4)}, 
{(0,4) \ar @{-o} (5,9)}, 
{(10,-9) \ar @{{*}-{*}} (10,-4)}, 
{(10,-4) \ar @{.{*}} (10,4)}, 
{(10,4) \ar @{-} (5,9)}, 
{(-1,-9) \ar @/^1mm/ @{-} ^k (-1,4)}, 
{(11,-9) \ar @/_1mm/ @{-} _{n-1} (11,4)}, 
}\ \right). 
\]
Therefore, by \eqref{eq:MZSV}, we obtain Ohno's relation \cite[Theorem 2]{O} 
\[\xi_k(n)=\zeta^\star(k+1,\underbrace{1,\ldots,1}_{n-1}). \]

\subsection{Mordell-Tornheim zeta values}
Next, we consider the values of the Mordell-Tornheim multiple zeta functions 
\cite{M}
\[\zeta_{MT,r}(s_1,\ldots,s_r;s)=\sum_{m_1,\ldots,m_r>0}
\frac{1}{m_1^{s_1}\cdots m_r^{s_r}(m_1+\cdots+m_r)^s}. \]
For positive integers $k_1,\ldots,k_r,k$, it is easy to show that 
\begin{equation}\label{eq:MT}
\zeta_{MT,r}(k_1,\ldots,k_r;k)=I\left(\xybox{
{(-8,-15) \ar @{{*}-o} (-8,-10)}, 
{(-8,-10) \ar @{.o} (-8,-2)}, 
{(-8,-2) \ar @{-o} (0,2)}, 
{(8,-15) \ar @{{*}-o} (8,-10)}, 
{(8,-10) \ar @{.o} (8,-2)}, 
{(8,-2) \ar @{-o} (0,2)}, 
{(0,2) \ar @{-o} (0,7)}, 
{(0,7) \ar @{.o} (0,15)}, 
{(-5,-5) \ar @{.} (5,-5)}, 
{(-9,-15) \ar @/^1mm/ @{-} ^{k_1} (-9,-2)}, 
{(9,-15) \ar @/_1mm/ @{-} _{k_r} (9,-2)}, 
{(1,2) \ar @/_1mm/ @{-} _k (1,15)}
}\right). 
\end{equation}
For example, 
\begin{align*}
I\left(\ \begin{xy}
{(-4,-4) \ar @{{*}-} (0,0)}, 
{(4,-4) \ar @{{*}-} (0,0)}, 
{(0,0) \ar @{o-o} (0,4)}\end{xy}\ \right)
&=\int_{1>t_1>t_2>0}\frac{dt_1}{t_1}\frac{dt_2}{t_2}
\biggl(\int_0^{t_2}\frac{du}{1-u}\biggr)
\biggl(\int_0^{t_2}\frac{dv}{1-v}\biggr)\\
&=\int_0^1\frac{dt_1}{t_1}\int_0^{t_1}\frac{dt_2}{t_2}
\Biggl(\sum_{m>0}\frac{t_2^m}{m}\Biggr)
\Biggl(\sum_{n>0}\frac{t_2^n}{n}\Biggr)\\
&=\sum_{m,n>0}\frac{1}{mn}
\int_0^1\frac{dt_1}{t_1}\int_0^{t_1}t_2^{m+n}\frac{dt_2}{t_2}\\
&=\sum_{m,n>0}\frac{1}{mn(m+n)}\int_0^1t_1^{m+n}\frac{dt_1}{t_1}\\
&=\sum_{m,n>0}\frac{1}{mn(m+n)^2}=\zeta_{MT,2}(1,1;2). 
\end{align*}

The identity \eqref{eq:MT} implies, in particular, 
the result by Bradley-Zhou \cite{BZ} 
that the Mordell-Tornheim zeta value $\zeta_{MT,r}(k_1,\ldots,k_r;k)$ is 
expressed as a finite sum of MZVs of weight $k_1+\cdots+k_r+k$ and depth $r$. 

\subsection{Certain zeta values of root systems of type $A$}
The third class of examples is a certain type of special values 
of zeta functions of root systems of type $A_N$, 
considered by Komori, Matsumoto and Tsumura \cite{KMT} 
in a study of shuffle relations of MZVs. 
Explicitly, these values are written as 
\[\zeta\biggl(p_1,\ldots,p_a;\begin{array}{c}
q_1,\ldots,q_b\\ r_1,\ldots,r_c\end{array}\biggr)
=\sum_{\substack{l_1>\cdots>l_a>m_1+n_1\\ m_1>\cdots>m_b>0\\ 
n_1>\cdots>n_c>0}}
\frac{1}{l_1^{p_1}\cdots l_a^{p_a}m_1^{q_1}\cdots m_b^{q_b}
n_1^{r_1}\cdots n_c^{r_c}}, \]
for three sequences $(p_1,\ldots,p_a)$, $(q_1,\ldots,q_b)$ and 
$(r_1,\ldots,r_c)$ of positive integers. 
To describe the corresponding diagram, we introduce an abbreviation: 
For a sequence $\bk=(k_1,\ldots,k_n)$ of positive integers, 
we write 
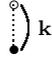
\[\begin{xy}
{(0,-3) \ar @{{*}.o} (0,3)}, 
{(1,-3) \ar @/_1mm/ @{-} _{\bk} (1,3)}
\end{xy}\]
for the vertical diagram 
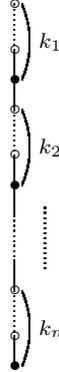
\[\begin{xy}
{(0,-24) \ar @{{*}-o} (0,-20)}, 
{(0,-20) \ar @{.o} (0,-14)}, 
{(0,-14) \ar @{-} (0,-10)}, 
{(0,-10) \ar @{.} (0,-4)}, 
{(0,-4) \ar @{-{*}} (0,0)}, 
{(0,0) \ar @{-o} (0,4)}, 
{(0,4) \ar @{.o} (0,10)}, 
{(0,10) \ar @{-{*}} (0,14)}, 
{(0,14) \ar @{-o} (0,18)}, 
{(0,18) \ar @{.o} (0,24)}, 
{(1,-24) \ar @/_1mm/ @{-} _{k_n} (1,-14)}, 
{(4,-3) \ar @{.} (4,-11)}, 
{(1,0) \ar @/_1mm/ @{-} _{k_2} (1,10)}, 
{(1,14) \ar @/_1mm/ @{-} _{k_1} (1,24)} 
\end{xy}\]
so that 
\[\zeta(\bk)=I\left(\ \begin{xy}
{(0,-3) \ar @{{*}.o} (0,3)}, 
{(1,-3) \ar @/_1mm/ @{-} _{\bk} (1,3)}
\end{xy}\right). \]
Using this notation, one can verify that 
\begin{equation}\label{eq:KMT}
\zeta(\mathbf{p};\mathbf{q};\mathbf{r})=I\left(\begin{xy}
{(-4,-8) \ar @{{*}.o} (-4,-2)}, 
{(-4,-2) \ar @{-} (0,2)}, 
{(4,-8) \ar @{{*}.o} (4,-2)}, 
{(4,-2) \ar @{-} (0,2)}, 
{(0,2) \ar @{{*}.o} (0,8)}, 
{(-5,-8) \ar @/^1mm/ @{-} ^{\mathbf{q}} (-5,-2)}, 
{(5,-8) \ar @/_1mm/ @{-} _{\mathbf{r}} (5,-2)}, 
{(1,2) \ar @/_1mm/ @{-} _{\mathbf{p}} (1,8)}
\end{xy}\right) 
\end{equation}
for $\mathbf{p}=(p_1,\ldots,p_a)$, $\mathbf{q}=(q_1,\ldots,q_b)$ 
and $\mathbf{r}=(r_1,\ldots,r_c)$.  

In \cite{KMT}, the following relation plays an important role: 
\begin{equation}\label{eq:KMTrel}
\begin{split}
&\zeta\biggl(p_1,\ldots,p_a;\begin{array}{c}
q_1,\ldots,q_b\\ r_1,\ldots,r_c\end{array}\biggr)\\
&=\sum_{j=0}^{q_1-1}\binom{r_1-1+j}{j}\,
\zeta\biggl(p_1,\ldots,p_a,r_1+j;\begin{array}{c}
q_1-j,q_2,\ldots,q_b\\ r_2,\ldots,r_c\end{array}\biggr)\\
&+\sum_{j=0}^{r_1-1}\binom{q_1-1+j}{j}\,
\zeta\biggl(p_1,\ldots,p_a,q_1+j;\begin{array}{c}
q_2,\ldots,q_b\\ r_1-j,r_2,\ldots,r_c\end{array}\biggr). 
\end{split}
\end{equation}

We point out that our expression \eqref{eq:KMT} implies this relation 
quite naturally. 
To do this, we denote by $X$ the 2-labeled poset indicated in \eqref{eq:KMT}, 
and name some vertices as follows: 
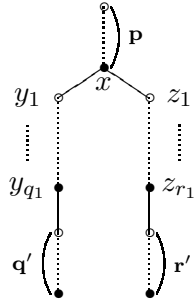
\[\begin{xy}
{(-6,-28) \ar @{{*}.o} (-6,-20)}, 
{(-6,-20) \ar @{-} (-6,-14)}, 
{(-6,-14) \ar @{{*}.o} (-6,-2)}, 
{(-10,-14)*++{y_{q_1}} \ar @{.} (-10,-2)*++{y_1}}, 
{(-6,-2) \ar @{-} (0,2)}, 
{(6,-28) \ar @{{*}.o} (6,-20)}, 
{(6,-20) \ar @{-} (6,-14)}, 
{(6,-14) \ar @{{*}.o} (6,-2)}, 
{(10,-14)*++{z_{r_1}} \ar @{.} (10,-2)*++{z_1}}, 
{(6,-2) \ar @{-} (0,2)}, 
{(0,2)*+!U{x} \ar @{{*}.o} (0,10)}, 
{(-7,-28) \ar @/^1mm/ @{-} ^{\mathbf{q}'} (-7,-20)}, 
{(7,-28) \ar @/_1mm/ @{-} _{\mathbf{r}'} (7,-20)}, 
{(1,2) \ar @/_1mm/ @{-} _{\mathbf{p}} (1,10)}
\end{xy}\]
where $\mathbf{q}'=(q_2,\ldots,q_b)$ and $\mathbf{r}'=(r_2,\ldots,r_c)$. 
By \eqref{eq:sum}, one has 
\begin{equation}\label{eq:KMT1}
I(X)=I(X_{y_{q_1}}^{z_{r_1}})+I(X^{y_{q_1}}_{z_{r_1}}). 
\end{equation}
In $X_{y_{q_1}}^{z_{r_1}}$, 
the inequalities $x>z_{r_1}>y_{q_1}$ and $x>y_1>\cdots>y_{q_1}$ hold, 
hence one can consider $q_1$ further refinements by imposing 
$y_j>z_{r_1}>y_{j+1}$ for $j=0,\ldots,q_1-1$ (here we write $y_0=x$). 
Thus the identity 
\begin{equation}\label{eq:KMT2}
I(X_{y_{q_1}}^{z_{r_1}})=\sum_{j=0}^{q_1-1}I(X_j) 
\end{equation}
holds, where $X_j$ is represented by the diagram 
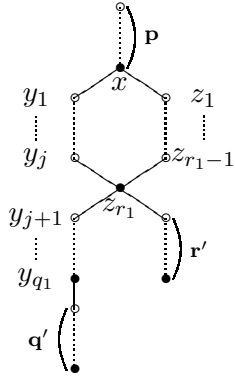
\[\begin{xy}
{(-6,-38) \ar @{{*}.o} (-6,-30)}, 
{(-6,-30) \ar @{-{*}} (-6,-26)}, 
{(-6,-26) \ar @{.o} (-6,-18)}, 
{(-11,-26)*+{y_{q_1}} \ar @{.} (-11,-18)*+{y_{j+1}}}, 
{(-6,-18) \ar @{-} (0,-14)}, 
{(0,-14)*+!U{z_{r_1}} \ar @{{*}-o} (-6,-10)}, 
{(-6,-10) \ar @{.o} (-6,-2)}, 
{(-11,-10)*+{y_j} \ar @{.} (-11,-2)*+{y_1}}, 
{(-6,-2) \ar @{-} (0,2)}, 
{(6,-26) \ar @{{*}.o} (6,-18)}, 
{(6,-18) \ar @{-} (0,-14)}, 
{(0,-14) \ar @{-o} (6,-10)}, 
{(6,-10) \ar @{.o} (6,-2)}, 
{(11,-10)*+{z_{r_1-1}} \ar @{.} (11,-2)*+{z_1}}, 
{(6,-2) \ar @{-} (0,2)}, 
{(0,2)*+!U{x} \ar @{{*}.o} (0,10)}, 
{(-7,-38) \ar @/^1mm/ @{-} ^{\mathbf{q}'} (-7,-30)}, 
{(7,-26) \ar @/_1mm/ @{-} _{\mathbf{r}'} (7,-18)}, 
{(1,2) \ar @/_1mm/ @{-} _{\mathbf{p}} (1,10)}
\end{xy}\]
Moreover, since there are $\binom{r_1-1+j}{j}$ ways to impose 
a total order on the white vertices $y_1,\ldots,y_j$, 
$z_1,\ldots,z_{r_1-1}$, one has 
\begin{equation}\label{eq:KMT3}
I(X_j)=\binom{r_1-1+j}{j}\,
\zeta\biggl(p_1,\ldots,p_a,r_1+j;\begin{array}{c}
q_1-j,q_2,\ldots,q_b\\ r_2,\ldots,r_c\end{array}\biggr). 
\end{equation}
The identities \eqref{eq:KMT2} and \eqref{eq:KMT3} 
expresses the first term of \eqref{eq:KMT1} 
as desired in \eqref{eq:KMTrel}. 
The second is obtained in the same way.

\begin{rem}
In \cite{KMT}, using partial fraction decompositions, 
the identity \eqref{eq:KMTrel} is proved with some variables 
(irrelevant to the decomposition) \emph{complex valued}, 
not necessarily positive integral. 
It seems difficult to apply our method in this paper 
to such functional relations. 
\end{rem}


\end{document}